\documentclass[12pt]{amsart}

\usepackage{verbatim}

\usepackage{amsthm}
\usepackage{amsfonts}
\usepackage{amsmath}
\usepackage{amssymb}
\usepackage{verbatim, color}

\usepackage{pictex}
\usepackage{url}

\begin{document}
 \newcounter{thlistctr}
 \newenvironment{thlist}{\
 \begin{list}%
 {\alph{thlistctr}}%
 {\setlength{\labelwidth}{2ex}%
 \setlength{\labelsep}{1ex}%
 \setlength{\leftmargin}{6ex}%
 \renewcommand{\makelabel}[1]{\makebox[\labelwidth][r]{\rm (##1)}}%
 \usecounter{thlistctr}}}%
 {\end{list}}

\thispagestyle{empty}

\newtheorem{Lemma}{\bf LEMMA}[section]
\newtheorem{Theorem}[Lemma]{\bf THEOREM}
\newtheorem{Claim}[Lemma]{\bf CLAIM}
\newtheorem{Corollary}[Lemma]{\bf COROLLARY}
\newtheorem{Proposition}[Lemma]{\bf PROPOSITION}
\newtheorem{Example}[Lemma]{\bf EXAMPLE}
\newtheorem{Fact}[Lemma]{\bf FACT}
\newtheorem{definition}[Lemma]{\bf DEFINITION}
\newtheorem{Notation}[Lemma]{\bf NOTATION}
\newtheorem{remark}[Lemma]{\bf REMARK}

\newcommand{\restrict}{\mbox{$\mid\hspace{-1.1mm}\grave{}$}}
\newcommand{\covers}{\mbox{$>\hspace{-2.0mm}-{}$}}
\newcommand{\covered}{\mbox{$-\hspace{-2.0mm}<{}$}}
\newcommand{\notcover}{\mbox{$>\hspace{-2.0mm}\not -{}$}}

\newcommand{\boldalpha}{\mbox{\boldmath $\alpha$}}
\newcommand{\boldbeta}{\mbox{\boldmath $\beta$}}
\newcommand{\boldgamma}{\mbox{\boldmath $\gamma$}}
\newcommand{\boldxi}{\mbox{\boldmath $\xi$}}
\newcommand{\boldlambda}{\mbox{\boldmath $\lambda$}}
\newcommand{\boldmu}{\mbox{\boldmath $\mu$}}

\newcommand{\barzero}{\bar{0}}

\newcommand{\sfq}{{\sf q}}
\newcommand{\sfe}{{\sf e}}
\newcommand{\sfk}{{\sf k}}
\newcommand{\sfr}{{\sf r}}
\newcommand{\sfc}{{\sf c}}
\newcommand{\restr}{\negmedspace\upharpoonright\negmedspace}

\title[De Morgan Semi-Heyting Algebras]{A Note on Regular De Morgan Semi-Heyting Algebras}
             
\author{Hanamantagouda P. Sankappanavar}
\date{\today}

\begin{abstract}
The purpose of this note is two-fold.  Firstly, we prove that the variety $\mathbf{RDMSH_1}$ of regular De Morgan semi-Heyting algebras of level 1 satisfies Stone identity and present (equational) axiomatizations for several subvarieties of $\mathbf{RDMSH_1}$.  Secondly, we give a concrete description of the lattice of subvarieties of the variety $\mathbf{RDQDStSH_1}$ of regular dually quasi-De Morgan Stone semi-Heyting algebras  
that contains $\mathbf{RDMSH_1}$.  Furthermore, we prove that every subvariety of $\mathbf{RDQDStSH_1}$, and hence of $\mathbf{RDMSH_1}$, has Amalgamation Property.  The note concludes with some open problems for further investigation.

\end{abstract}

\maketitle

\thispagestyle{empty}

\section{{\bf Introduction}} \label{SA}
  
Semi-Heyting algebras were introduced by us in \cite{Sa07} as an abstraction of Heyting algebras.  They share several important properties with Heyting algebras, such as distributivity, pseudocomplementedness, and so on.  On the other hand, interestingly, there are also semi-Heyting algebras, which, in some sense, are ``quite opposite'' to Heyting algebras.  For example, the identity $0 \to 1 \approx 0$, as well as the commutative law $x \to y \approx y \to x$, hold in some semi-Heyting algebras.  The subvariety of commutative semi-Heyting algebras was defined in \cite{Sa07} and is  further investigated in \cite{Sa10}.  

Quasi-De Morgan algebras were defined in \cite{Sa87a} as a common abstraction of De Morgan algebras and distributive $p$-algebras.  
In \cite{Sa12}, expanding semi-Heyting algebras  by adding a dual quasi-De Morgan operation, we introduced the variety 
$\mathbf{DQDSH}$ of dually quasi-De Morgan semi-Heyting algebras as a common generalization of De Morgan Heyting algebras (see \cite{Sa87} and \cite{Mo80}) and dually pseudocomplemented Heyting algebras (see \cite{Sa85}) so that we could settle an old conjecture of ours.  

The concept of regularity has played an important role in the theory of pseudocomplemented De Morgan algebras (see \cite{Sa86}).  Recently, in \cite{Sa14} and  \cite{Sa14a}, we inroduced and examined the concept of regularity in the context of  
$\mathbf{DQDSH}$ and gave an explicit description of (twenty five) simple algebras in the (sub)variety $\mathbf{DQDStSH_1}$ of regular dually quasi-De Morgan 
Stone semi-Heyting algebras of level 1.  The work in \cite{Sa14} and  \cite{Sa14a} led us to conjecture that the variety $\mathbf{RDMSH_1}$ of regular De Morgan algebras satisfies Stone identity.

The purpose of this note is two-fold.  Firstly, we prove that the variety $\mathbf{RDMSH_1}$ of regular De Morgan semi-Heyting algebras of level 1 satisfies Stone identity, thus settlieng the above mentioned conjecture affirmatively.  As applications of this result and the main theorem of \cite{Sa14}, we present (equational) axiomatizations for several subvarieties of $\mathbf{RDMSH_1}$.    Secondly, we give a concrete description of the lattice of subvarieties of the variety $\mathbf{RDQDStSH_1}$ of regular dually quasi-De Morgan Stone semi-Heyting algebras, of which $\mathbf{RDMSH_1}$ is a subvariety.  Furthermore, we prove that every subvariety of $\mathbf{RDQDStSH_1}$, and hence of $\mathbf{RDMSH_1}$, has Amalgamation Property.  The note concludes with some open problems for further investigation.

\vspace{1cm}
\section{\bf {Dually Quasi-De Morgan Semi-Heyting Algebras}} \label{SB}

The following definition is taken from \cite{Sa07}. 

An algebra ${\mathbf L}= \langle L, \vee ,\wedge ,\to,0,1 \rangle$
is a {\it semi-Heyting algebra} if \\
 $\langle L,\vee ,\wedge ,0,1 \rangle$ is a bounded lattice and ${\mathbf L}$ satisfies:
\begin{enumerate}
\item[{\rm(SH1)}] $x \wedge (x \to y) \approx x \wedge y$
\item[{\rm(SH2)}] $x \wedge(y  \to z) \approx x \wedge ((x \wedge y) \to (x \wedge z))$
\item[{\rm(SH3)}] $x \to x \approx 1$.
\end{enumerate}
Let ${\mathbf L}$ be a semi-Heyting algebra and, for $x \in {\mathbf L}$, let $x^*:=x \to 0$.  ${\mathbf L}$ is a {\it Heyting algebra} if ${\mathbf L}$ satisfies:
\begin{enumerate}
\item[{\rm(SH4)}] $(x \wedge y) \to y \approx 1$.
\end{enumerate}
${\mathbf L}$ is a {\it commutative semi-Heyting algebra} if ${\mathbf L}$ satisfies:
\begin{enumerate}
\item[{\rm(Co)}] $x  \to y \approx y \to x$.
\end{enumerate}
${\mathbf L}$ is a {\it Boolean semi-Heyting algebra} if ${\mathbf L}$ satisfies:
\begin{enumerate}
\item[{\rm(Bo)}] $x \lor x^{*} \approx 1$.  
\end{enumerate}
${\mathbf L}$ is a {\it Stone semi-Heyting algebra} if ${\mathbf L}$ satisfies:
\begin{enumerate}
\item[{\rm(St)}] $x^* \lor x^{**} \approx 1$.  
\end{enumerate}

Semi-Heyting algebras are distributive and pseudocomplemented,
with $a^*$ as the pseudocomplement of an element $a$.  We will use these
and other properties (see \cite{Sa07}) of semi-Heyting algebras,
frequently without explicit mention, throughout this paper.

The following definition is taken from \cite{Sa12}.
\begin{definition}
An algebra ${\mathbf L}= \langle L, \vee ,\wedge ,\to, ', 0,1
\rangle $ is a {\it semi-Heyting algebra with a dual quasi-De Morgan
operation} or {\it dually quasi-De Morgan semi-Heyting algebra}
\rm{(}$\mathbf {DQDSH}$-algebra, for short\rm{)}  if\\
 $\langle L, \vee ,\wedge ,\to, 0,1 \rangle $ is a semi-Heyting algebra, and
${\mathbf L}$ satisfies:
 \begin{itemize}
    \item[(a)]  $0' \approx 1$ and $1' \approx 0$
     \item[(b)] 
      $(x \land y)' \approx x' \lor y'$
    \item[(c)]
    $(x \lor y)''  \approx x'' \lor y''$     
    \item[(d)]  
   $x'' \leq x$.
\end{itemize}
\noindent 
Let $\mathbf{L} \in \mathbf {DQDSH}$.  Then ${\bf L}$ is a {\it dually Quasi-De Morgan Stone semi-Heyting algebra} {\rm(}$\mathbf{DQDStSH}$-algebra{\rm)} if ${\bf
L}$ satisfies (St). 
 $\mathbf {L}$ is a {\it De Morgan semi-Heyting algebra} or {\it symmetric semi-Heyting algebra} {\rm(}$\mathbf{DMSH}$-algebra{\rm)} if ${\bf
L}$ satisfies:
\begin{itemize}
\item[(DM)]  $x'' \approx x$.
\end{itemize}
$\mathbf{L}$ is a {\it  dually pseudocomplemented semi-Heyting algebra}
{\rm(}$\mathbf {DPCSH}$-algebra if $\mathbf{L}$ satisfies: 
 \begin{itemize}
 \item[ (PC)] $x \lor x' \approx 1$.
\end{itemize}

The varieties of $\mathbf {DQDSH}$-algebras, $\mathbf{DQDStSH}$-algebras, 
 $\mathbf{DMSH}$-algebras  and $\mathbf {DPCSH}$-algebras are denoted, respectively, by $\mathbf {DQDSH}$, $\mathbf{DQDStSH}$, 
 $\mathbf{DMSH}$ and $\mathbf {DPCSH}$.  Furthermore, $\mathbf {DMcmSH}$ denotes the subvariety of $\mathbf {DMSH}$ defined by the commutative identity (Co), and
  $\mathbf {DQDBSH}$ denotes the one defined by (Bo). 
  \end{definition} 
If the underlying semi-Heyting algebra of a $\mathbf{DQDSH}$-algebra
is a Heyting algebra we denote the algebra by $\mathbf{DQDH}$-algebra, and the corresponding variety is denoted by $\mathbf{DQDH}$.

In the sequel, $a'{^*}'$ will be denoted by $a^+$, for $a \in \mathbf{L} \in \mathbf{DQDSH}$. 
The following lemma will often be used without explicit reference to it.  Most of the items in this lemma were proved in \cite{Sa12}, and the others are left to the reader.
\begin{Lemma} \label{2.2}
Let ${\mathbf L} \in \mathbf{DQDSH}$ and let $x,y, z \in L$.  Then
\begin{enumerate}
 \item[{\rm(i)}]  $1'^{*}=1$
\item[{\rm(ii)}] $x \leq y$ implies $x' \geq y'$
\item[{\rm(iii)}] $(x \land y)'^{*}=x'^{*} \land y'^{*}$
\item[{\rm(iv)}] $ x''' = x'$
\item[{\rm(v)}]  $(x \lor y)' = (x'' \lor y'')'$ \label{C.934}
\item[{\rm(vi)}] $(x \lor y)' = (x'' \lor y)'$ \label{12662}
\item[{\rm(vii)}] $x  \leq (x  \lor y) \to x$ \label{B.63}
\item[{\rm(viii)}]  $x \land [(x \lor y) \to z] = x \land z$.
\end{enumerate}
\end{Lemma}
Next, we describe some examples of $\mathbf {DQDSH}$-algebras by expanding the semi-Heyting algebras given in Figure 1.  These will play a crucial role in the rest of the note.
\vspace{.4cm}  

\setlength{\unitlength}{.9cm} \vspace*{1cm}
\begin{picture}(5,2)

\put(1.7,1){\circle*{.15}}         

\put(1.7,2.3){\circle*{.15}}       

\put(1.3,.9){$0$}               

\put(1.3,2.1){$1$}                 

\put(.2,1.5){$\bf 2$ \ :}

\put(1.7,1){\line(0,1){1.2}}

\put(2.6,1.5){\begin{tabular}{c|cc}
$\to$ \ & \ 0 \ & \ 1  \\ \hline
 0 & 1 & 1  \\ 
 1 & 0 & 1  \\ 
\end{tabular}
}

\put(8.7,1){\circle*{.15}}

\put(8.7,2.3){\circle*{.15}}

\put(8.3,.9){$0$}      

\put(8.3,2.1){$1$}

\put(7.0,1.5){${\bf {\bar{2}}}$ \ :}

\put(8.7,1){\line(0,1){1.2}}

\put(9.5,1.5){\begin{tabular}{c|cc}
 $\to$ \ & \ 0 \ & \ 1 \\ \hline
  0 & 1 & 0 \\
  1 & 0 & 1  \\
  
\end{tabular}
}
\end{picture}\\

\vspace{.2cm}

\setlength{\unitlength}{.9cm} \vspace*{.7cm}       
\begin{picture}(5, 1)                 
\put(1.5, .2){\circle*{.15}}          

\put(1.5,1.0){\circle*{.15}}

\put(1.5,1.8){\circle*{.15}}

\put(1.1, .1){$0$}                      

\put(1.1,.9){$a$}

\put(1.1,1.7){$1$}

\put(.1,.7){${\bf L}_1$ \ :}

\put(1.5, .2){\line(0,1){1.5}}

\put(2.2,1){\begin{tabular}{c|ccc}
 $\to$ \ & \ 0 \ & \ $a$ \ & \ 1 \\ \hline
  0 & 1 & 1 & 1 \\
  $a$ & 0 & 1 & 1 \\
  1 & 0 & $a$ & 1
\end{tabular}
}

\put(8.2, .3){\circle*{.15}}                           

\put(8.2,1.1){\circle*{.15}}

\put(8.2, 1.8){\circle*{.15}}

\put(7.8, .2){$0$}                                     

\put(7.8,.9){$a$}

\put(7.8,2.1){$1$}

\put(6.7,.7){${\bf L}_2$ \ :}                        

\put(8.2, .3){\line(0,1){1.5}}                        

\put(8.9,1){\begin{tabular}{c|ccc}                     
 $\to$ \ & \ 0 \ & \ $a$ \ & \ 1 \\ \hline
  0 & 1 & $a$ & 1 \\
  $a$ & 0 & 1 & 1 \\
  1 & 0 & $a$ & 1
\end{tabular}
}
\end{picture}\\

\vspace{1cm}
\setlength{\unitlength}{.9cm} \vspace*{.7cm}   
\begin{picture}(7,.3)

\put(1.5, .2){\circle*{.15}}          

\put(1.5,1.0){\circle*{.15}}

\put(1.5,1.8){\circle*{.15}}

\put(1.1, .1){$0$}                      

\put(1.1,.9){$a$}

\put(1.1,1.7){$1$}

\put(.1,.7){${\bf L}_3$ \ :}

\put(1.5, .2){\line(0,1){1.5}}

\put(2.2,1){\begin{tabular}{c|ccc}
 $\to$ \ & \ 0 \ & \ $a$ \ & \ 1 \\ \hline
  0 & 1 & 1 & 1 \\
  $a$ & 0 & 1 & $a$ \\
  1 & 0 & $a$ & 1
\end{tabular}
}

\put(8.2, .3){\circle*{.15}}                           

\put(8.2,1.1){\circle*{.15}}

\put(8.2, 1.8){\circle*{.15}}

\put(7.8, .2){$0$}                                     

\put(7.8,.9){$a$}

\put(7.8,2.1){$1$}

\put(6.7,.7){${\bf L}_4$ \ :}                        

\put(8.2, .3){\line(0,1){1.5}}

\put(8.9,1){\begin{tabular}{c|ccc}
 $\to$ \ & \ 0 \ & \ $a$ \ & \ 1 \\ \hline
  0 & 1 & $a$ & 1 \\
  $a$ & 0 & 1 & $a$ \\
  1 & 0 & $a$ & 1
\end{tabular}
}
\end{picture} \\

\setlength{\unitlength}{.9cm} \vspace*{.7cm}
\begin{picture}(7, 1.7)

\put(1.5, .2){\circle*{.15}}          

\put(1.5,1.0){\circle*{.15}}

\put(1.5,1.8){\circle*{.15}}

\put(1.1, .1){$0$}                      

\put(1.1,.9){$a$}

\put(1.1,1.7){$1$}

\put(.1,.7){${\bf L}_5$ \ :}

\put(1.5, .2){\line(0,1){1.5}}

\put(2.2,1){\begin{tabular}{c|ccc}
 $\to$ \ & \ 0 \ & \ $a$ \ & \ 1 \\ \hline
  0 & 1 & $a$ & $a$ \\
  $a$ & 0 & 1 & 1 \\
  1 & 0 & $a$ & 1
\end{tabular}
}


\put(8.2, .3){\circle*{.15}}                           

\put(8.2,1.1){\circle*{.15}}

\put(8.2, 1.8){\circle*{.15}}

\put(7.8, .2){$0$}                                     

\put(7.8,.9){$a$}

\put(7.8,2.1){$1$}

\put(6.7,.7){${\bf L}_6$ \ :}                        

\put(8.2, .3){\line(0,1){1.5}}

\put(8.9,1){\begin{tabular}{c|ccc}
 $\to$ \ & \ 0 \ & \ $a$ \ & \ 1 \\ \hline
  0 & 1 & 1 & $a$ \\
  $a$ & 0 & 1 & 1 \\
  1 & 0 & $a$ & 1
\end{tabular}
}
\end{picture}\\
\ \\ \ \\ \ \\ 

\setlength{\unitlength}{.9cm} \vspace*{.7cm}
\begin{picture}(7,0.3)

\put(1.5, .2){\circle*{.15}}          

\put(1.5,1.0){\circle*{.15}}

\put(1.5,1.8){\circle*{.15}}

\put(1.1, .1){$0$}                      

\put(1.1,.9){$a$}

\put(1.1,1.7){$1$}

\put(.1,.7){${\bf L}_7$ \ :}

\put(1.5, .2){\line(0,1){1.5}}

\put(2.2,1){\begin{tabular}{c|ccc}
 $\to$ \ & \ 0 \ & \ $a$ \ & \ 1 \\ \hline
  0 & 1 & $a$ & $a$ \\
  $a$ & 0 & 1 & $a$ \\
  1 & 0 & $a$ & 1
\end{tabular}
}


\put(8.2, .3){\circle*{.15}}                           

\put(8.2,1.1){\circle*{.15}}

\put(8.2, 1.8){\circle*{.15}}

\put(7.8, .2){$0$}                                     

\put(7.8,.9){$a$}

\put(7.8,2.1){$1$}

\put(6.7,.7){${\bf L}_8$ \ :}                        

\put(8.2, .3){\line(0,1){1.5}}

\put(8.9,1){\begin{tabular}{c|ccc}
 $\to$ \ & \ 0 \ & \ $a$ \ & \ 1 \\ \hline
  0 & 1 & 1 & $a$ \\
  $a$ & 0 & 1 & $a$ \\
  1 & 0 & $a$ & 1
\end{tabular}
}
\end{picture}\\
\  \\ \ \\
\medskip
\setlength{\unitlength}{.9cm} \vspace*{.7cm}    
\begin{picture}(7,2.0)

\put(1.5, .2){\circle*{.15}}          

\put(1.5,1.0){\circle*{.15}}

\put(1.5,1.8){\circle*{.15}}

\put(1.1, .1){$0$}                      

\put(1.1,.9){$a$}

\put(1.1,1.7){$1$}

\put(.1,.7){${\bf L}_9$ \ :}

\put(1.5, .2){\line(0,1){1.5}}

\put(2.2,1){\begin{tabular}{c|ccc}
 $\to$ \ & \ 0 \ & \ $a$ \ & \ 1 \\ \hline
  0 & 1 & 0 & 0 \\
  $a$ & 0 & 1 & 1 \\
  1 & 0 & $a$ & 1
\end{tabular}
}

\put(8.2, .3){\circle*{.15}}                           

\put(8.2,1.1){\circle*{.15}}

\put(8.2, 1.8){\circle*{.15}}

\put(7.8, .2){$0$}                                     

\put(7.8,.9){$a$}

\put(7.8,2.1){$1$}

\put(6.7,.7){${\bf L}_{10}$ \ :}                        

\put(8.2, .3){\line(0,1){1.5}}      

\put(8.9,1){\begin{tabular}{c|ccc}
 $\to$ \ & \ 0 \ & \ $a$ \ & \ 1 \\ \hline
  0 & 1 & 0 & 0 \\
  $a$ & 0 & 1 & $a$ \\
  1 & 0 & $a$ & 1
\end{tabular}
}

\end{picture}\\
\setlength{\unitlength}{.9cm} \vspace*{.3cm}
\begin{picture}(7,4)

\put(.4,2.5){$\bf D_1$ \ :}

\put(7.0,2.5){$\bf D_2$ \ :}

\put(1.5,2.4){\begin{tabular}{c|cccc}                          
 $\to$ \ & \ 0 \ & \ $1$ \ & \ $a$\ &\ $b$ \\ \hline
  0 & 1 & 0 & $b$ & $a$ \\
  $1$ & 0 & 1 & $a$ & $b$ \\
  $a$ & $b$ & $a$ & 1 &0\\
  $b$ & $a$ & $b$ & 0 & 1\\
\end{tabular}
}
\put(8.0,2.4){\begin{tabular}{c|cccc}
 $\to$ \ & \ 0 \ & \ $1$ \ & \ $a$\ &\ $b$ \\ \hline
  0 & 1 & 1 & 1 & 1 \\
  $1$ & 0 & 1 & $a$ & $b$ \\
  a & $b$ & $1$ & 1 & $b$ \\
  $b$ & $a$ & 1 & $a$ & 1\\
\end{tabular}
}
                          
 \end{picture}  
\setlength{\unitlength}{.9cm} \vspace*{.3cm}    

 \begin{picture}(7, 3) \\

 \put(.4,2.0){$\bf D_3$ \ :}                   
                            
\put(2.2,2.4){\begin{tabular}{c|cccc}                  
 $\to$ \ & \ 0 \ & \ $1$ \ & \ $a$\ &\ $b$ \\ \hline
  0 & 1 & $a$ & 1 & $a$ \\
  $1$ & 0 & 1 & $a$ & $b$ \\
  $a$ & $b$ & $a$ & 1 &0\\
  $b$ & $a$ & 1 & $a$ & 1\\
\end{tabular}
}
\vspace{.5cm}
\put(6,0.2){Figure 1}    
\end{picture}

\vspace{.3cm}

Let $\mathbf {2^e}$ and $\mathbf{\bar{2}^e}$ be the expansions of the semi-Heyting algebras $\mathbf {2}$ and $\mathbf{\bar{2}}$ (shown in Figure 1) by adding the unary operation $'$ such
that $0'=1$, $1'=0$.\par

Let $\mathbf{L}^{dp}_i$, $i=1, \ldots,10$, denote the expansion of the semi-Heyting algebra
$\mathbf{L}_i$ (shown in Figure 1) by adding the unary operation $'$ such that $0'=1$, $1'=0$, and $a'=1$.\par  

Let $\mathbf{L}^{dm}_i$, $i=1, \ldots, 10$,
 denote the expansion of ${\bf L}_i$ (in Figure 1) by adding the unary operation $'$ such that $0'=1$, $1'=0$, and $a'=a$. 

We Let $\mathbf{C^{dp}_{10}} := \{\mathbf{L}^{dp}_i: i =1, \ldots,10 \}$ and 
 $\mathbf{C^{dm}_{10}} := \{\mathbf{L}^{dm}_i: i =1, \ldots,10 \}$.   We also let 
$ \mathbf{C_{20}} := \mathbf{C^{dm}_{10}} \cup \mathbf{C^{dp}_{10}}$.        

Each of the three $4$-element 
algebras $\mathbf{D_1}$, $\mathbf{D_2}$ and
$\mathbf{D_3}$ 
has its lattice reduct as the Boolean lattice with the universe $\{0,a,b,1\}$, $b$ being the complement of
$a$, has the operation $\to$ as defined in Figure 1,  
 and has the unary operation $'$ defined as follows: $a' = a$, $b' = b$, $0'=1$, $1'=0$.  For the variety $\mathbf{V(D_1, D_2, D_3)}$ generated by $\mathbf{\{D_1, D_2, D_3 \}}$, it was shown in \cite{Sa12} that $\mathbf{V(D_1, D_2, D_3)} = \mathbf{DQDBSH}$. 

The following is a special case of Definition 5.5 in \cite{Sa12}.  Let $x'{^*}'{^*} := x{^{2(}}{'^{*)}}$.  Note that $x{^{2(}}{'^{*)}} \leq x$ in a 
$\mathbf{DMSH}$-algebra.  
\vspace{-.46cm} 
                                                   
\begin{definition}\label{5.5}
The subvariety $\mathbf {DMSH_1}$ of level 1
of $\mathbf {DMSH}$ is defined by the identity:
$x \land x'^* \land x^{2{\rm(}}{'^{*{\rm{)}}}} \approx x \land x'^*,$ or equivalently, by the identity:
\begin{itemize}
\item[\rm (L1)]  $(x \land x'^*)'^* \approx x \land x'^*$.  
\end{itemize}
\end{definition} 
It follows from \cite{Sa12} that  
the variety 
$\mathbf {DMSH_1}$, is a discriminator variety. 
We note here that the algebras described above in Figure 1 
are actually in $\mathbf {DMSH_1}$.

\vspace{1cm}

\section{Regular De Morgan Semi-Heyting algebras of level 1}

Recall that $a^+:= a'{^*}'$ in $\mathbf{L} \in \mathbf{DMSH_1}$. 
\begin{definition}
Let $\mathbf{L} \in \mathbf{DMSH_1}$.  Then $\mathbf{L}$  {\it is regular}  
 if $\mathbf{L}$ satisfies the following identity:
\begin{itemize}
\item[(R)]  $ x \land x^+ \leq y \lor y^*$.
\end{itemize}
The variety of regular $\mathbf{DMSH_1}$-algebras will be denoted by 
$\mathbf{RDMSH_1}$.  
\end{definition}

In the rest of this section, $\mathbf{L}$ denotes an $\mathbf{RDMSH_1}$-algebra and $x,y \in L$.  The following lemmas lead us to prove that $\mathbf{RDMSH_1}$ satisfies (St). 

\begin{Lemma} \label{regB}

$(x \lor x{^*}')^* =  x' \land x^*$.  
\end{Lemma}

\begin{proof}
\begin{eqnarray*}
x' \land x^* &=& x' \land x''^*\\
                       &=& (x' \land x'{'^*})'^*  \quad \text{ by (L1)}\\
                       &= & (x'' \lor x''{^*}')^* \\
                       &=&  (x \lor x{^*}')^*, \quad \text{ since $x''=x$}. 
\end{eqnarray*}                       
\end{proof}

\begin{Lemma} \label{regC}
 $x \lor x^* \lor x{^*}' = 1$.
\end{Lemma}

\begin{proof}
\begin{eqnarray*}
 x \lor x^* \lor x{^*}'  &=& (x{^*}' \land x' \land x^*)'  \quad \text{ by (DM)}\\
                                    &=& [x{^*}' \land (x \lor x{^*}')^*]' \quad \text{ by Lemma \ref{regB}}\\
                                    &=& (x{^*}' \land 0)' \quad \text{ by Lemma \ref{2.2}(viii)}\\
                                     &=& 0'\\
                                     &=& 1.
\end{eqnarray*}    
\end{proof}

\begin{Lemma} \label{regE}  We have
\begin{itemize}
\item[ ]  $x \land (x^+ \lor y \lor y^*) = x \land (y \lor y^*)$. 
\end{itemize}
\end{Lemma}

\begin{proof}
\begin{eqnarray*}
x \land (y \lor y^*) &=& x \land [(x \land x^+) \lor (y \lor y^*)] \quad \text{ by (R)}\\
                                &=& (x \land x^+) \lor [x \land (y \lor y^*)] \\
                               & =& x \land [x^+ \lor y \lor y^*].
\end{eqnarray*}
\end{proof}

\begin{Lemma}\label{regF}  Let $x \neq 1$.  Then
  $x \leq x'$. 
\end{Lemma}

\begin{proof}
Since $x \neq 1$, we have $x \land x'^*=0$ by (L1).  So, 
\begin{eqnarray*}
x \land x'   &=&(x \land x')\lor (x\land x'^*) \\
                   &=& x \land (x' \lor x'^*) \\   
                    &= & x \land (x^+ \lor x' \lor x'^*) \quad \text{by Lemma \ref{regE}}\\
                   & = & x \land 1 \quad \text{ by Lemma \ref{regC}} 
\end{eqnarray*}
 So, $x \leq x'$.
\end{proof}

\begin{Lemma} \label{regG} Let $x^* \neq 0$.  Then
 $x \lor x{^*} =1$. 
\end{Lemma}

\begin{proof}
Since $x^* \neq 0$, we have $x{^*}' \neq1$, so $x{^*}' \leq x{^*} $ by Lemma \ref{regF} and (DM), 
implying
 $x \lor x^* =1$ by Lemma \ref{regC}.  
\end{proof}

\begin{Theorem}\label{regH}
 Let $\mathbf{L} \in \mathbf{RDMSH_1}$.  Then $\mathbf{L} \models x^* \lor x^{**} \approx 1$.
\end{Theorem}

\begin{proof}
Let $a \in L$.    
If $a^* = 0$, Then the theorem is trivially true.  So, we can assume that $a^* \neq 0$.  Then  $a \lor a^* = 1$, in view of the preceding lemma.  The conclusion is now immediate.
\end{proof}

Recall from \cite{Sa12} that the subvariety $\mathbf {DMSH_2}$ of level 2
of $\mathbf {DMSH}$ is defined by the identity:
$x \land x'^* \land x^{2{\rm(}}{'^{*{\rm{)}}}} \approx x \land x'^*\land x^{2{\rm(}}{'^{*{\rm{)}}}} \land x^{3{\rm(}}{'^{*{\rm{)}}}},$ or equivalently, by the identity:\\

{\rm (L2)}  $(x \land x'^*)^{2{\rm(}}{'^{*{\rm{)}}}} \approx (x \land x'^*)^{{\rm(}}{'^{*{\rm{)}}}}$. 
\begin{remark}
The above theorem fails in $\mathbf{RDMSH_2}$, as the following example shows:
\newpage
\setlength{\unitlength}{.7cm} \vspace*{1cm}
\begin{picture} (8,1)  

\put(6,0){\circle*{.15}}  
\put(6.2,0){$e$}

\put(4,0){\circle*{.15}}  
\put(3.5,0){$d$}  

\put(5,1){\circle*{.15}} 
\put(5.2,-1.1){$c$}  


\put(4,-2){\circle*{.15}}  
\put(3.5, -2){$a$}  

\put(6,-2){\circle*{.15}}  
\put(6.2, -2.1){$b$}  

\put(5,-3){\circle*{.15}} 
\put(5.2,-3.3){$0$}  

\put(5,1){\circle*{.15}} 
\put(5.2, 1.1){$1$}  

\put(4,0){\line(1,1){1}}
\put(5,-1){\line(1,1){1}}
\put(4,-2){\line(1,1){1}}
\put(5,-3){\line(1,1){1}}

\put(4,0){\line(1,-1){1}}
\put(5,-1){\line(1,-1){1}}
\put(4,-2){\line(1,-1){1}}
\put(5,1){\line(1,-1){1}}

\put(5,-4.3){Figure 2}
\end{picture}

\end{remark}

\vspace{3.5cm}

\section{Applications}

Let $\mathbf{V(K)}$ denote the variety generated by the class $\mathbf{K}$ of algebras.

The following corollary is immediate from Theorem \ref{regH} and Corollary 3.4(a) of \cite{Sa14a}, 
and hence is an improvement on Corollary 3.4(a) of \cite{Sa14a}.


\begin{Corollary} We have\\
\begin{itemize}
\item[{\rm (a)}]
$\mathbf{RDMSH_1} = \mathbf{RDMStSH_1} = \mathbf{V(C_{10}^{dm})} \lor \mathbf{V(D_1, D_2, D_3)}$\\
\item[{\rm (b)}]
$\mathbf{RDMH_1} = \mathbf{RDMStH_1} = \mathbf{V(L_1^{dm})} \lor \mathbf{V(D_2)}$\\
\item[{\rm (c)}]
$\mathbf{RDMcmSH1} = \mathbf{V(L_{10}^{dm}, D_1)} = \mathbf{V(L_{10}^{dm})} \lor \mathbf{V(D_1)}$.
\end{itemize}
\end{Corollary}


Let $L \in \mathbf{DMSH_1}$.  We say $\mathbf{L}$ is pseudocommutative if \\
$\mathbf{L} \models (x \to y)^* = (y \to x)^*$.\\

\begin{Corollary}
Let $\mathbf{V}$ be a subvariety of $ \mathbf{RDMSH_1}$.  Then $\mathbf{V}$ is pseudocommutative iff
 $\mathbf{V}= \mathbf{V(L_9^{dm}, L_{10}^{dm}, D_1)}$.
\end{Corollary}

\begin{proof}
It suffices, in view of (a) of the preceding corollary, to verify that 
$\mathbf{L_9^{dm}, L_{10}^{dm}}$, and $\mathbf{D_1}$ satisfy the pseudocommutative law,
while the rest of the simples in $\mathbf{RDMSH_1}$ do not.
\end{proof}
The proofs of the following corollaries are similar.
\begin{Corollary}
 The variety $\mathbf{V(L_9^{dm}, L_{10}^{dm}, D_1})$ is also defined, modulo $\mathbf{RDMSH1}$, by 
\begin{itemize}
\item[ ]  $x^* \to y^* \approx  y^* \to x^*$.
\end{itemize}
\end{Corollary}

\begin{Corollary}
 The variety $\mathbf{V(L_1^{dm}, L_2^{dm}, L_3^{dm}, L_4^{dm},  D_2, D_3})$ is defined, modulo $\mathbf{RDMSH1}$, by 
\begin{itemize}
\item[ ]  $(0 \to 1)^+  \to (0 \to 1){^*}'^* \approx  0 \to 1$.
\end{itemize}
\end{Corollary}

It  was proved in \cite{Sa12} that $\mathbf{V(D_1, D_2, D_3}) = \mathbf{DQDBSH}$.  Here are some more bases for $\mathbf{V(D_1, D_2, D_3})$.

\begin{Corollary}
Each of the following identities is a base for
 the variety $\mathbf{V(D_1, D_2, D_3})$ modulo $\mathbf{RDMSH1}$: 
\begin{itemize}
\item[(1)]  $x \to y  \approx  y^* \to x^*$  {\rm (Law of contraposition)}
\item[(2)]  $x \lor (y \to z) \approx  (x \lor y) \to (x \lor z)$
\item[(3)]   $ [\{x \lor (x \to y^*)\} \to (x \to y^*)] \lor (x \lor y^*)=1$.\\
\end{itemize}
\end{Corollary}

\begin{Corollary} 
 The variety \\
 $\mathbf{V(L_1^{dm}, L_2^{dm}, L_5^{dm}, L_6^{dm}, L_9^{dm}, D_1, D_2, D_3})$ is defined, modulo\\
  $\mathbf{RDMSH1}$, by 
\begin{itemize}
\item[ ]  $x \to y^*  \approx  y \to x^*$.
\end{itemize}
\end{Corollary}

\begin{Corollary}
 The variety $\mathbf{V(L_7^{dm}, L_8^{dm}, L_9^{dm}, L_{10}^{dm}, D_1, D_2, D_3})$ is defined, modulo $\mathbf{RDMSH1}$, by 
\begin{itemize}
\item[ ]  $x \lor  (x \to y)  \approx  x \lor [(x \to y) \to 1]$.
\end{itemize}
\end{Corollary}

\begin{Corollary}
 The variety $\mathbf{V(L_7^{dm}, L_8^{dm},  D_2})$ is defined, modulo $\mathbf{RDMSH1}$, by 
\begin{itemize}
\item[(1)]  $x \lor  (x \to y)  \approx  x \lor [(x \to y) \to 1]$
\item[(2)] $ (0 \to 1)^{**} \approx 1$.
\end{itemize}
\end{Corollary}

\begin{Corollary} 
 The variety $\mathbf{V(2^e, L_7^{dm}, L_8^{dm}, L_9^{dm},  L_{10}^{dm}})$ is\\ defined, modulo $\mathbf{RDMSH1}$, by 
\begin{itemize}
\item[(1)]  $x \lor  (x \to y)  \approx  x \lor [(x \to y) \to 1]$
\item[(2)]  $x{^*}' \approx  x^{**}$.
\end{itemize}
\end{Corollary}

\begin{Corollary}
 The variety \\
 $\mathbf{V(2^e,  L_9^{dm},  L_{10}^{dm}})$ is defined, modulo $\mathbf{RDMSH1}$, by 
\begin{itemize}
\item[(1)]  $x \lor  (x \to y)  \approx  x \lor [(x \to y) \to 1]$
\item[(2)]  $x{^*}' \approx  x^{**}$
\item[(3)] $ (0 \to 1) \lor (0 \to 1)^* \approx 1$.
\end{itemize}
\end{Corollary}

\begin{Corollary}
 The variety \\
 $\mathbf{V(L_9^{dm},  L_{10}^{dm}})$ is defined, modulo 
 $\mathbf{RDMSH1}$, by 
\begin{itemize}
\item[(1)]  $x \lor  (x \to y)  \approx  x \lor [(x \to y) \to 1]$
\item[(2)]  $x{^*}' \approx  x^{**}$
\item[(3)] $ (0 \to 1)^{*} \approx 1$.
\end{itemize}
\end{Corollary}

\begin{Corollary}
 The variety \\
 $\mathbf{V(L_1^{dm}, L_2^{dm}, L_3^{dm}, L_4^{dm}, L_5^{dm}, L_6^{dm}, L_7^{dm}, L_8^{dm}})$ is defined, modulo \\
 $\mathbf{RDMSH1}$, by 
\begin{itemize}
\item[(1)] $x{^*}' \approx  x^{**}$
\item[(2)] $ (0 \to 1)^{**} \approx 1$.
\end{itemize}
\end{Corollary}

\begin{Corollary}
 The variety $\mathbf{V(L_1^{dm}, L_2^{dm}, L_3^{dm}, L_4^{dm}, D_2})$ is \\ defined, modulo $\mathbf{RDMSH1}$, by 
\begin{itemize}
\item[(1)]  $(0 \to 1) \lor (0 \to 1)^* \approx 1$
\item[(2)] $ (0 \to 1)^{**} \approx 1$.
\end{itemize}
\end{Corollary}

\begin{Corollary}
 The variety $\mathbf{V(L_1^{dm}, L_3^{dm}, D_1, D_2, D_3})$ is \\
 defined, modulo $\mathbf{RDMSH1}$, by 
\begin{itemize}
\item[(1)]  $x \lor (y \to x)  \approx (x \lor y) \to x$
\item[(2)] $(0 \to 1) \lor (0 \to 1)^* \approx 1 $.
\end{itemize}
\end{Corollary}

\begin{Corollary}
 The variety $\mathbf{V(L_1^{dm},  L_3^{dm}, D_2})$ is defined, modulo $\mathbf{RDMSH1}$, by 
\begin{itemize}
\item[(1)]  $x \lor (y \to x)  \approx (x \lor y) \to x$
\item[(2)]  $(0 \to 1)^{**}=1$
\item[(3)] $(0 \to 1) \lor (0 \to 1)^* \approx 1 $.
\end{itemize}
\end{Corollary}

\begin{Corollary}
 The variety $\mathbf{V(L_1^{dm}, L_2^{dm}, L_8^{dm}, D_1, D_2, D_3})$ is defined, modulo $\mathbf{RDMSH1}$, by 
\begin{itemize}
\item[ ]  $y \lor (y \to (x \lor  y)) \approx (0 \to x) \lor (x \to y)$.
\end{itemize}
\end{Corollary}

\begin{Corollary}
 The variety $\mathbf{V(L_8^{dm}, D_1, D_2, D_3)}$ is defined, modulo $\mathbf{RDMSH1}$, by 
\begin{itemize}
\item[ ]  $x \lor [y \to (0 \to (y \to x))] \approx x \lor y \lor (y \to x)$.
\end{itemize}
\end{Corollary}

 $\mathbf{V(D_2)}$ was axiomatized in \cite{Sa12}.  Here are some more bases for it.

\begin{Corollary}
Each of the following identities is an equational base for 
 $\mathbf{V(D_2)}$, mod  $\mathbf{RDMH_1}$: \\
  
\begin{itemize}
\item[(1)]  $ y \to [0 \to (y \to x)] \approx  y \lor (y \to x)$\\

\item[(2)]  $x \lor (y \to z) \approx (x \lor  y) \to (x \lor z)$\\

\item[(3)] $x \lor [y \to (y \to x)^*] \approx x \lor y \lor (y \to x) $\\

\item[(4)]  $[\{x \lor (x \to y^*)\} \to (x \to y^*)] \lor x \lor y^* \approx 1.$
\end{itemize}

\end{Corollary}


$\mathbf{V(D_1)}$ was axiomatized in \cite{Sa12}.  Here are more bases for it.

\begin{Corollary}
Each of the following identities is an equational base for 
 $\mathbf{V(D_1)}$, mod  $\mathbf{RDMcmSH_1}$:\\

\begin{itemize}

\item[(1)] $y \lor (y \to (x \lor  y)) \approx (0 \to x) \lor (x \to y)$ \\ 

\item[(2)]   $x \lor [y \to (y \to x)^*] \approx x \lor y \lor (y \to x)$\\

\item[(3)]   $[\{x \lor (x \to y^*)\} \to (x \to y^*)] \lor x \lor y^* \approx 1$\\

\item[(4)]  $x \lor (y \to z) \approx (x \lor  y) \to (x \lor z)$.\\
 \end{itemize}
\end{Corollary}


\begin{Corollary}
 The variety $\mathbf{V(L_1^{dm}, L_2^{dm}, L_3^{dm}, D_1, D_2, D_3))}$ is  defined, mod 
 $\mathbf{RDMSH_1}$, by 
\begin{itemize}
\item[ ]  $x \to (y \to z)=y \to (x \to z).$
\end{itemize}
\end{Corollary}

\begin{Corollary}
 The variety $\mathbf{V(C_{10}^{dm})}$ is defined, mod \\
 $\mathbf{RDMSH_1}$, by 
\begin{itemize}
\item[ ]  $x \land x' \leq y \lor y'$  {\rm (Kleene identity)}.
\end{itemize}
\end{Corollary}

\begin{Corollary}
 The variety $\mathbf{V(L_{10}^{dm})}$ is defined, mod $\mathbf{RDMSH_1}$, by 
\begin{itemize}
\item[(1)]  $x \land x' \leq y \lor y'$  {\rm (Kleene identity)}
\item[(2)]  $x \to y \approx y \to x$.
\end{itemize}
\end{Corollary}

\vspace{1cm}

\section{Lattice of subvarieties of $\mathbf{RDQDStSH_1}$ }

We now turn to describe the lattice of subvarieties of $\mathbf{RDQDStSH_1}$ which contains $\mathbf{RDMSH_1}$ in view of Theorem \ref{regH}.   For this purpose we need the following theorem which is proved in \cite{Sa14}.

\begin{Theorem}\label{Th}
Let $\mathbf{L} \in \mathbf{RDQDStSH1}$.  Then TFAE:
\begin{itemize}
\item[(1)] $L$ is simple\\
\item[(2)] $L$ is subdirectly irreducible\\
\item[(3)] $\mathbf{L} \in \mathbf{\{2^e, \bar{2}^e\}} \cup \mathbf{C_{20}} \cup \mathbf{\{D_1, D_2, D_3 \}}$.
\end{itemize}
\end{Theorem}

Let $\mathcal{L}$ denote the lattice of subvarieties of $\mathbf{RDQDStSH_1}$.
$\mathbf{T}$ denotes the trivial variety, and, for $n$ a positive integer, $\mathbf{B_n}$ denotes the $n$-atom Boolean lattice.   We also let $\mathbf{1+B}$ denote 
the lattice obtained by adding a new least element $0$ to the Boolean lattice $\mathbf{B}$.

\begin{Theorem}
$\mathcal{L} \cong \mathbf{(1+B_9)} \times  \mathbf{(1+B_5)} \times \mathbf{B_9}$.
\end{Theorem}

\begin{proof}
Let $S_1$ := $\mathbf{\{L_i^{dm}: i =1,2,3,4\}} \cup 
 \mathbf{\{L_i^{dp}: i =1,2,3,4\}} \cup \mathbf{\{D_2 \}}$, \\
  $S_2$ := $\mathbf{\{L_i^{dm}: i =9, 10\}} \cup 
 \mathbf{\{L_i^{dp}: i =9, 10 \}} \cup \mathbf{\{D_1 \}}$, and \\
 $S_3$:= $\mathbf{\{L_i^{dm}: i =5,6,7,8 \}} \cup 
 \mathbf{\{L_i^{dp}: i = 5,6,7,8 \}} \cup \mathbf{\{D_3 \}}$.
Observe that each of the simples in $S_1$
contains $\mathbf{2^e}$.  Let us first look at the interval $\mathbf{[V(2^e), V(S_1)]}$.
Since each algebra in $S_1$ is an atom in this interval, 
we can conclude that the interval is a 9-atom Boolean lattice; thus the 
interval $\mathbf{[T, V(S_1)]}$
 is isomorhic to $\mathbf{1+B_9}$.  Similarly, since each of the simples in $S_2$
contains $\mathbf{\bar{2}^e}$, it is clear that the interval $\mathbf{[T, V(S_2)]}$ 
is isomorphic to $\mathbf{1+B_5}$.  
Likewise, since each of the simples in $S_3$ 
has only one subalgebra, namely the trivial algebra, the interval
$\mathbf{[T, S_3]}$ 
is isomprphic to $\mathbf{B_9}$.  Observe that the the intersection of the subvarieties $\mathbf{V(S_1)}$, $\mathbf{V(S_2)}$ and $\mathbf{V(S_3)}$ is $\mathbf{T}$ 
and their join is $\mathbf{RDQDSH_1}$ in $\mathcal{L}$.  It, therefore, follows that $\mathcal{L}$ is isomorphic to 
$\mathbf{(1+B_9)} \times  \mathbf{(1+B_5)} \times \mathbf{B_9}$.
\end{proof}

\begin{Corollary}
The lattice of subvarieties of $\mathbf{RDMSH_1}$ is isomorphic to $\mathbf{(1+B_5)} \times  \mathbf{(1+B_3)} \times \mathbf{B_5}$.
\end{Corollary}

\begin{Corollary}
The lattice of subvarieties of $\mathbf{RDPCSH_1}$ is isomorphic to $\mathbf{(1+B_4)} \times  \mathbf{(1+B_2)} \times \mathbf{B_4}$.
\end{Corollary}
Similar formulas can be obtained for other subvarieties of $\mathbf{RDQDSH_1}$.
\vspace{1cm}
\section{Amalgamation}

We now examine the Amalgamation Property for subvarities of the variety $\mathbf{RDQDStSH_1}$.  For this purpose we need the following theorem from \cite{GrLa71}.

\begin{Theorem}  Let K be an equational class of algebras satisfying the Congruence Extension Property, and let every subalgebra of each subdirectly irreducible algebra in K be subdirectly irreducible.  Then K satisfies the Amalgamation Property if and only if whenever A, B, C are subdirectly irreducible algebras in K with A a common subalgebra of B and C, the amalgam (A; B, C) can be amalgamated in K.
\end{Theorem}

\begin{Theorem}
Every subvariety of $\mathbf{RDQDStSH_1}$ has the Amalgamation Property.
\end{Theorem}

\begin{proof}
It follows from \cite{Sa12} that $\mathbf{RDQDStSH_1}$ has CEP.  Also, it follows from Theorem \ref{Th} that every subalgebra of each subdirectly irreducible (= simple) algebra in $\mathbf{RDQDStSH_1}$ is subdirectly irreducible.  Therefore, in each subvariety $\mathbf{V}$ of $\mathbf{RDQDStSH_1}$, we need only consider an amalgam $(A: B, C)$, where $A, B, C$ are simple in $\mathbf{RDQDStSH_1}$ and $A$ a subalgebra of $B$ and $C$.  Then it is not hard to see, in view of the description of simples in $\mathbf{RDQDStSH_1}$ given in Theorem \ref{Th}, that $(A: B, C)$ can be amalgamated in $\mathbf{V}$.  
\end{proof}

\vspace{1cm}
\section{Concluding Remarks and Open Problems}

We know from \cite{Sa12} that every simple algebra in $\mathbf{RDQDH_1}$ is quasiprimal.

Of all the 25 simple algebras in $\mathbf{RDQDStSH_1}$, $\mathbf{2^e}$, 
 $\mathbf{\bar{2}^e}$, and $\mathbf{L_i}, i = 5,6,7,8$, and $\mathbf{D_3}$ are primal algebras and the rest are semiprimal algebras.  
We now mention some open problems for further research.\\

Problem 1: For each variety $\mathbf{V(L)}$, where $\mathbf{L}$ is a simple algebra in $\mathbf{RDMSH_1}$ (except $\mathbf{V(2^e)}$), find a Propositional Calculus $\mathbf{P(V)}$ such that the equivalent algebraic semantics for $\mathbf{P(V)}$ is $\mathbf{V(L)}$ (with $1$ as the designated truth value, using $\to$ and $'$ as implication and negation respectively). (For the variety $\mathbf{V(2^e)}$, the answer is, of course, well known: Classical Propositional Calculus.) \par
 We think such (many-valued) logics will be of interest in computer science and  in switcing circuit theory. \\ 

Problem 2: Describe simples in the variety of pseudocommutative $\mathbf{RDQDStSH_1}$-algebras.\\

Problem 3:  Find equational bases for the remaining subvarieties  of  $\mathbf{RDMSH_1}$.\\

Problem 4: Let $\mathbf{RDmsStSH_1}$ denote the subvariety of $\mathbf{DQDStSH_1}$ defined by: $(x \lor y)' \approx x' \land y'$.  Describe simples in  
 $\mathbf{RDmsStSH_1}$.\\

\small

\ \\
\ \ \ \ \ \ \ \ \\
Department of Mathematics\\
State University of New York\\
New Paltz, NY 12561\\
\
\\
sankapph@newpaltz.edu\\
\ \

\end{document}